\documentclass{amsart}
\usepackage[utf8]{inputenc}
\usepackage{amsthm,amsmath,amssymb,amscd}
\usepackage[colorlinks=true,allcolors=blue!80!black]{hyperref}
\usepackage[parfill]{parskip}  
\usepackage{graphics}
\usepackage{tikz, tikz-cd,circuitikz}
\usetikzlibrary{matrix, snakes, patterns, shadows.blur, backgrounds}
\usepackage[all]{xy}
\usepackage{enumerate}
\usepackage{bbm}

\theoremstyle{plain}
\newtheorem{thm}{Theorem}[section]
\newtheorem{lem}[thm]{Lemma}

\newtheorem*{thm*}{Theorem}
\newtheorem{abcthm}{Theorem}

\theoremstyle{definition}
\newtheorem{defn}[thm]{Definition}

\newtheorem*{exam*}{Example}
\newtheorem{exam}[thm]{Example}
\newtheorem{rem}[thm]{Remark}

\def\Z{\mathbb Z}

\def\R{\mathbb R}
\def\Q{\mathbb Q}

\def\D{D}

\def\link{\mathcal{E}}
\def\rlink{\mathcal{R}}
\def\prlink{\mathcal{PR}}
\def\Sa{\mathbb S}
\def\Hy{\mathbb H}

\newcommand{\Mob}{\text{M\"ob}}
\DeclareMathOperator\Int{Int}

\DeclareMathOperator\im{Im}

\DeclareMathOperator\Isom{Isom}

\DeclareMathOperator\emb{Emb}
\DeclareMathOperator\Diff{Diff}
\DeclareMathOperator{\Homeo}{Homeo}

\DeclareMathOperator\remb{REmb}

\DeclareMathOperator{\arc}{arc}
\DeclareMathOperator{\coll}{Coll}
\DeclareMathOperator{\Gr}{Gr}

\DeclareMathOperator\SO{SO}
\DeclareMathOperator\SU{SU}
\DeclareMathOperator{\Or}{O}
\DeclareMathOperator{\mob}{\rm{M\ddot{o}b}}

\newcommand{\fakeenv}{} 
\newenvironment{restate}[2]  
{
  \renewcommand{\fakeenv}{#2}   
  \theoremstyle{plain}
  \newtheorem*{\fakeenv}{#1~\ref{#2}} 
  \begin{\fakeenv}  
}
{ \end{\fakeenv} }

\definecolor{champagne}{rgb}{0.97, 0.91, 0.81}
\definecolor{burlywood}{rgb}{0.55, 0.71, 0.0}
\definecolor{cfblue}{rgb}{0.39, 0.58, 0.93}
\definecolor{amber}{rgb}{1.0, 0.75, 0.0}

\title{The embedding space of a Hopf link}

\author{Rachael Boyd}
\address{School of Mathematics and Statistics, University of Glasgow, Glasgow G12 8QQ, UK}
\email{rachael.boyd@glasgow.ac.uk}
\urladdr{https://www.maths.gla.ac.uk/~rboyd/} 

\author{Corey Bregman}
\address{Department of Mathematics, Tufts University, Medford, MA 02155, USA}
\email{corey.bregman@tufts.edu}
\urladdr{https://sites.google.com/view/cbregman} 
\makeatletter
\@namedef{subjclassname@2020}{\textup{2020} Mathematics Subject Classification}
\makeatother
\subjclass[2020]{58D10, 
55P15
}
\begin{document}
\maketitle
\begin{abstract}We study the unparametrised smooth embedding space of a Hopf link in~$\R^3$, and prove that it is is homotopy equivalent to the closed 3-manifold $S^3/\Q_8$. As an intermediate step in the proof, we show that the inclusion of the subspace of round embeddings is a homotopy equivalence. We provide analogous results for the unparametrised smooth embedding space of a Hopf link in~$S^3$, which we show is homotopy equivalent to $\R P^2\times \R P^2$.
\end{abstract}

\section{Introduction} \label{sec: intro}

Let $H\colon S^1\sqcup S^1\hookrightarrow \R^3$ be the smooth 2-component link whose image is the union of the unit circle centered at the origin in the $xy$-plane and the unit circle centered at $(1,0,0)$ in the $xz$-plane. The isotopy class of $H$ gives the simplest example of a nontrivial 2-component link, and is known as the \emph{Hopf link}. We study the homotopy type of the unparametrised smooth embedding space of the Hopf link in~$\R^3$, i.e.~the connected component of
$$
\emb (S^1 \sqcup S^1, \R^3)/\Diff (S^1 \sqcup S^1)
$$
containing $H$. Denote this space by~$\link(H)$.
Our main result shows that $\link(H)$ is homotopy equivalent to a finite CW complex and provides an explicit finite CW model. Regarding $S^3$ as the Lie group $\SU(2)$ of unit quaternions, let $\Q_8$ denote the subgroup $\langle\pm1,  \pm i,\pm j,\pm k\rangle$ and let~$S^3/\Q_8$ be the prism-manifold quotient. We prove the following theorem.

\begin{abcthm}\label{abcthm: homotopy type}
The embedding space $\link(H)$ is homotopy equivalent  to $S^3/\Q_8.$ 
\end{abcthm}

The corresponding \emph{round} embedding space, which we denote~$\rlink(H)$, is the subspace of $\link(H)$ for which each component is a round circle lying in an affine hyperplane $\R^2\subset \R^3$. To prove Theorem \ref{abcthm: homotopy type}, we first reduce to studying this subspace.

\begin{abcthm}\label{abcthm: round space equivalence}
The inclusion map $\rlink(H)\hookrightarrow \link(H)$ is a homotopy equivalence.
\end{abcthm}

In the final section of the paper we prove analogous results for the unparametrised embedding space of a Hopf link in $S^3$, which we denote by $\link(H,S^3)$. In this setting, the inclusion of a round subspace is also a homotopy equivalence, as we show in Theorem~\ref{thm: equiv great and smooth hl in s3}, and has the homotopy type of an explicit finite manifold.

\begin{abcthm}\label{abcthm: S3 homotopy type}
    The embedding space $\link(H,S^3)$ is homotopy equivalent to $\R P^2\times \R P^2.$
\end{abcthm}

Recall that Hatcher \cite[Appendix]{Hatcher} showed that the unparametrised embedding space of an unknot in $S^3$ is homotopy equivalent to the space of great circles in $S^3$, which can be identified with the Grassmannian $\Gr_2(\R^4)$. 
To prove Theorem \ref{abcthm: S3 homotopy type}, we show that $\link(H,S^3)$ is homotopy equivalent to a quotient of $\Gr_2(\R^4)$ by the involution which exchanges $V$ and $V^\perp$. It is interesting that, up to homotopy, the space of Hopf links is naturally a quotient of the space of unknots in this way.

Damiani and Kamada computed the fundamental group of~$\rlink(H)$ to be the quaternion group~$\Q_8$ \cite{DamianiKamada}, and we independently obtain that the (smooth) \emph{motion group} $\pi_1(\link(H))$ is isomorphic to $\Q_8$. This agrees with Goldsmith's computation of the \emph{topological motion group} $\pi_1(\Homeo(\R^3), \Homeo(\R^3,L))$ of torus links \cite{Goldsmith} where we view $H$ as the torus link $T(2,2)$. We therefore conclude that the smooth and topological motion groups are the same. By~\cite[Corollary B]{BoydBregman} the motion group in $S^3$ agrees with the motion group in $\R^3$.

In \cite{BrendleHatcher}, Brendle--Hatcher studied the unparametrised embedding space $\link(U_n)$, where $U_n$ is the~$n$-component unlink in~$\R^3$. They also prove that the inclusion of the subspace $\rlink(U_n)$ of round embeddings is a homotopy equivalence. As a consequence, one knows that $\link(U_n)$ has the homotopy type of a finite-dimensional manifold. Our work is inspired by this, in the setting of a single Hopf link.

In previous work we study the homotopy type of the unparametrised embedding space $\link(L)$ of a \emph{split link} $L$ in $\R^3$ \cite{BoydBregman}. We show that $\link(L)$ is homotopy equivalent to a space with extra structure, and thus provide an iterated semi-direct product formula for the {motion group} $\pi_1(\link(L))$. When $L$ is a union of unknot and Hopf link components, the results of \cite{BoydBregman} are combined with Theorem~\ref{abcthm: round space equivalence} to derive a formula for the motion group, given in \cite[Example 4.12]{BoydBregman}.

\subsection*{Organisation}
In Section \ref{sec: preliminaries} we set up definitions and notation, and collect results on the homotopy type of some embedding spaces for later use. We prove Theorem~\ref{abcthm: round space equivalence} in Section~\ref{sec: round equiv smooth} and Theorem~\ref{abcthm: homotopy type} in Section~\ref{sec: homotopy type}. In Section~\ref{sec: great hopf links} we consider the space of Hopf links in $S^3$ and prove the two analogous theorems. 

\subsection*{Acknowledgments}
This work was motivated by a talk of Damiani at the 2019 KaBiN conference in Trondheim, Norway.
We would like to thank Kiyoshi Igusa, Jean-Fran\c{c}ois Lafont and Arunima Ray for helpful conversations. 
The first author was partially supported by the Max Planck Institute for Mathematics in Bonn, ERC grant No.~756444, and EPSRC Fellowship No.~EP/V043323/1 and EP/V043323/2. 
The second author was supported by NSF grant DMS-2401403.

We would like to thank the Isaac Newton Institute for funding the satellite programme \emph{Topology, representation theory and higher structures} based at Gaelic College, Sabhal Mòr Ostaig, Isle of Skye, where part of this work was undertaken.

\section{Setting up notation and preliminary lemmas}\label{sec: preliminaries}

In this section we establish notation and define the spaces which we will work with throughout the paper. We then collect some important results concerning the topology of embedding spaces that we will need.

\subsection{Definitions and notation}
Let $\emb(M,N)$ be the space of smooth embeddings of~$M$ into~$N$, with the $C^\infty$ topology, and let $\emb_\partial(M,N)$ be the subspace of smooth embeddings of~$M$ into~$N$ that restrict to a fixed embedding on~$\partial M$, i.e.~$\partial M$ is a pointwise fixed subset of~$\partial N$. All embeddings we consider will be proper, i.e.~map boundaries to boundaries and interiors to interiors. 

Let $\Diff(M)$ be the diffeomorphism group of $M$ equipped with the $C^\infty$ topology. For $M\subset N$ let $\Diff_M(N)$ be the subgroup of diffeomorphisms of $N$ which restrict to the identity map on $M$, and let $\Diff(N,M)$ be the subgroup of diffeomorphisms of $N$ which fix $M$ setwise, i.e.~$f(M)=M$. If $M$ has boundary, let $\Diff_\partial(M)=\Diff_{\partial M}(M)$.

\begin{exam}\label{example:Basic Hopf}
The basic model for the Hopf link is the embedding  $H\colon \sqcup_2 S^1\hookrightarrow \R^3$ with image 
\begin{eqnarray*}
\text{Component }1: & x^2+y^2=1,& z=0\\
\text{Component }2: & (x-1)^2 +z^2=1,& y=0.
\end{eqnarray*}
See Figure~\ref{fig:H}. We will refer to any link that is isotopic to $H$ as a Hopf link.
\end{exam}
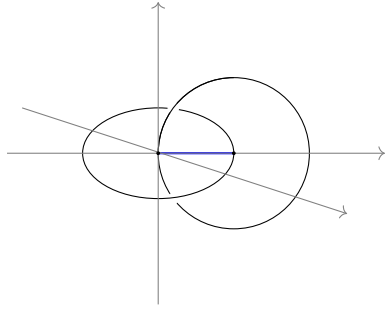
\begin{figure}[h!]
        \centering
        \begin{tikzpicture}
            \draw (1,0) ellipse (1 and 1);                  
            \filldraw[white](0.2,-0.6) circle (2pt);
            \draw (0,0) ellipse (1 and 0.6);                
            \filldraw[white](0.2,0.6) circle (2pt);
            \draw (1,1) arc[start angle=90, end angle=180, radius=1];
            \draw[thick,blue] (0,0)--(1,0);
            \draw[->,gray] (-2,0)--(3,0);
            \draw[->,gray] (0,-2)--(0,2);
            \draw[->,gray] (-1.8,0.6)--(2.5,-0.8);
            \filldraw(0,0) circle (.5pt);
            \filldraw(1,0) circle (.5pt);
        \end{tikzpicture}
        \caption{The basic model for the Hopf link, $H$, which serves as the basepoint in $\link(H)$.}
        \label{fig:H}
    \end{figure}
    
\begin{defn}\label{defn: HL emb space}
	Let $\link(H)$ be the connected component of
	$$
	\emb (S^1 \sqcup S^1, \R^3)/\Diff (S^1 \sqcup S^1)
	$$
	containing $H$. Working modulo the diffeomorphism group implies that each embedded~$S^1$ is unoriented and unparametrised, and the components of the link are unordered. We therefore refer to  $\link(H)$ as the \emph{unparametrised embedding space of Hopf links in $\R^3$}.
\end{defn}

The image of an embedding $\rho \in \link(H)$ is a submanifold of~$\R^3$ isotopic to $H$. Throughout this paper we will not differentiate between this submanifold and the embedding.
The embedding $H$ in Example \ref{example:Basic Hopf} will serve as the basepoint for $\link(H)$. Note that by definition each component of $H$ is a round circle lying in an affine hyperplane of $\R^3$.

\begin{defn}\label{defn: round unknot}
A \emph{round unknot} is an unparametrised embedding~\[\phi \in \emb (S^1, \R^3)/\Diff (S^1)\] that determines a Euclidean circle in an affine hyperplane~$\R^2\subseteq \R^3$. A \emph{round Hopf link} is~$\rho \in \link(H)$ such that each component of $\rho$ is a round unknot.  Let~$\rlink(H)$ be the subspace of~$\link(H)$ consisting of round Hopf links, and call this the \emph{round embedding space}.
\end{defn}

Observe that every round unknot bounds a canonical flat disc~$\D^2$ in~$\R^3$. The canonical discs spanned by the two components of a round Hopf link $\rho\in \rlink(H)$ intersect in a straight arc, with one endpoint on each unknot component. 
We call this (unoriented) arc the \emph{arc of intersection} of the Hopf link, and denote it~$\arc(\rho)$. It is shown in blue in Figure~\ref{fig:H}.

The embedding $H$ from Example \ref{example:Basic Hopf} is round. Hence $H$ will also serve as a basepoint for $\rlink(H)$. The link $H$ also has the following additional properties which we record here.
\begin{itemize}
\item Each component of $H$ has radius 1.
\item The canonical discs spanned by the components of $H$ are perpendicular.  
\item The endpoints of the arc of intersection $\arc(H)$ are the center points of the canonical disks spanned by the two components.
\end{itemize}

\subsection{Rounding unknots}The  following lemma is equivalent to the Smale conjecture (see (15) in the Appendix to \cite{Hatcher}) and  will be used to make the components of Hopf links round.
\begin{lem}\label{lem: rounding one unknot}
    The space $\rlink(U)$ of round embeddings of the unknot in $\R^3$ is homotopy equivalent to the space $\link(U)$ of smooth embeddings.
\end{lem}
\subsection{Straightening unknotted arcs}
Let $B^3$ be the closed 3-ball.  An equator divides $\partial B^3$ into two open hemispheres $S^2_+$ and $S^2_-$.  Consider the space \[\emb_U((I,\{1\},\{0\}),(B^3,S^2_+,S^2_-))\]
of proper embeddings of $I=[0,1]$ into $B^3$ which send $\{1\}$ to $S^2_+$ and $\{0\}$ into $S^2_-$, where the subscript $U$ indicates that the embeddings are isotopic to an unknotted arc.
\begin{lem}\label{lem: NP SP arc space contractible}
    The space of embeddings $\emb_U((I,\{1\},\{0\}),(B^3,S^2_+,S^2_-))$ is contractible.
\end{lem}
\begin{proof}
    The forgetful map $\emb_U((I,\{1\},\{0\}),(B^3,S^2_+,S^2_-))\rightarrow S^2_+\times S^2_-$ which records the image of the endpoints is a fibration. By a statement equivalent to the Smale conjecture the fiber is contractible (see (6) in the Appendix to \cite{Hatcher}). Since $S^2_+\times S^2_-$  is contractible, we conclude $\emb_U((I,\{1\},\{0\}),(B^3,S^2_+,S^2_-))$ is also contractible.
\end{proof}

\subsection{Contractibility of collars}

Let $M$ be a smooth compact manifold with boundary and let $F\subset \partial M$ be a compact codimension 0 submanifold (possibly with boundary). 
\begin{defn}
    A \emph{collar of $F$} is an embedding of $\chi\colon F\times I\rightarrow M$ such that $\chi(F\times \{0\})=F\subset \partial M$ and $\chi(F\times(0,1])\subset\Int(M)$. Denote by $\coll(F)$ the space of collars of $F$.
\end{defn}
When $\partial F\neq\emptyset$, we regard $F\times I$ as a manifold with corners.  On the other hand, if $\partial F=\emptyset$ (and hence $F$ is a component of $\partial M$), then the above notion of collar agrees with the usual one. The following result is due to Cerf \cite[5.2.1, Corollaire 1]{Cerf}.

\begin{lem}\label{lem:Boundary-Collar-Contractible}Let $F\subset \partial M$ be a codimension 0 compact submanifold. The space of collars $\coll(F)$ is contractible.
\end{lem}

\section{Relating the round and smooth embedding spaces} \label{sec: round equiv smooth}
In this section we prove our intermediate theorem: 

\begin{restate}{Theorem}{abcthm: round space equivalence}
The inclusion map $\rlink(H)\hookrightarrow \link(H)$ is a homotopy equivalence.
\end{restate}

Our proof takes inspiration from Hatcher's work \cite{Hatcher76, Hatcher99}. Since $\link(H)$ is a Fr\'echet manifold (see \cite{BinzFischer78,Michor80}) and thus has the homotopy type of a CW complex \cite{Milnor,Palais}, it is enough to prove that the inclusion $\rlink(H)\hookrightarrow\link(H)$ induces isomorphisms on all homotopy groups. Equivalently, we must show that the relative homotopy groups $\pi_k(\link(H),\rlink(H))$ vanish for all $k\geq 0$.

\subsection{Notation and first reductions} Since the proof will occupy the entire section we first set up some notation used throughout the proof, and describe a preliminary step. To this end, suppose $f\colon (D^k,\partial D^k) \to (\link(H),\rlink(H))$ is a $D^k$-parametrised family of Hopf links representing a relative homotopy class $[f]\in \pi_k(\link(H),\rlink(H))$. 
\begin{itemize}
    \item The standard round Hopf link $H$ of Example \ref{example:Basic Hopf} is the basepoint for both $\link(H)$ and $\rlink(H)$, so for the basepoint $t_0\in \partial D^k=S^{k-1}$, $f(t_0)=H$.
    \item We make a choice of labelling of the link components at the basepoint and denote them $L_1=L_1(t_0)$ and $L_2=L_2(t_0)$. Since $D^k$ is contractible this yields a continuous choice of labelling $L_1(t)$ and $L_2(t)$ of the link $f(t)=(L_1(t),L_2(t))$ for each $t\in D^k$.
\end{itemize}  

\begin{lem}\label{lem:One-Component-Round}
    The map $f\colon (D^k,\partial D^k) \to (\link(H),\rlink(H))$ is homotopic rel $\partial D^k$ to a map such that $L_1(t)$ is round for all $t\in D^k$.
\end{lem}
\begin{proof}
     The restriction $L_1(t)$ of $f$ to the first component is a $D^k$-family of unknots  representing a relative homotopy class in $\pi_k(\link(U),\rlink(U))$. By Lemma~\ref{lem: rounding one unknot}, this relative homotopy group is trivial, so we can homotope $L_1(t)$ rel $\partial D^k$ into $\rlink(U)$. Extend this homotopy by isotopy extension to the second component $L_2(t)$. Note that $L_2(t)$ remains round and unchanged for $t\in S^{k-1}$, since the homotopy restricts to the identity on $\partial D^k$, and thus we can assume that the isotopy extension acts as the identity.
\end{proof}

\subsection{An isotopy procedure for a single Hopf link}\label{sec:Isotopy-Removing-Intersections}

Before jumping into the proof of Theorem \ref{abcthm: round space equivalence}, we first describe the isotopy we will use to make a single embedded Hopf link round. This description can be viewed as a $\pi_0$-version of the isotopy in the proof, which we will extend to families of embeddings parametrised by a disc $D^k$. 

Suppose $L$ is a smoothly embedded Hopf link with components $L=(L_1, L_2)$, where $L_1$ is round and has radius $r_1$. Let $D_1$ be the canonical flat disc bounded by $L_1$ and let $D_2$ be any smooth disc bounded by $L_2$. In general, $D_1$ and $L_2$ may intersect in multiple points and may not be transverse. We will now describe a procedure for isotoping $L_2$ to a new embedding $\overline{L}_2$ without changing $L_1$, so that $D_1$ and $\overline{L}_2$ intersect transversely in a single point. 

Without loss of generality, assume that $L_1$ is the circle $(\frac{x_1}{r_1})^2+(\frac{y}{r_1})^2=1$ in the $xy$-plane.  Consider the family of ellipsoids $E_h$ of the form $(\frac{x_1}{r_1})^2+(\frac{y}{r_1})^2+(\frac{z}{h})^2=1$ for $h\in(0,\infty)$. Thus,  $E_h$ is an ellipsoid with equator  $L_1$ and that passes through the points $(0,0,\pm h)$ on the $z$-axis. We decompose $E_h=H^+_h\cup_{L_1}H^-_h$, where $H^\pm_h$ denotes the top and bottom hemispheres (technically hemellipsoids), respectively. Lastly, let $B_h$ be the solid convex `ball' whose boundary is $E_h$. The north pole (NP) is $(0,0,h)\in H_h^+$, while the south pole (SP) is $(0,0,-h)\in H^-_h$. 

 By residuality of transverse embeddings~\cite[Chapter 3, Theorem 2.1]{Hirsch}, $D_2$ will be transverse to $E_h$ for a dense open subset of values $h\in (0,\infty)$, which we indicate by writing $E_h\pitchfork D_2$. In this case, $E_h\cap D_2$ will be a disjoint union of finitely many properly embedded arcs and circles in $D_2$. The example in Figure \ref{fig:IntersectionPattern} shows the arc of intersection.
 
\begin{defn} If $E_h\pitchfork D_2$, the submanifold $K\subset E_h\cap D_2\subset D_2 $ consisting only of intersection arcs is called an \emph{intersection pattern}.
\end{defn} 

The following technical lemma will be used in the proof of Theorem~\ref{abcthm: round space equivalence}.

\begin{lem}\label{lem:Isotopy-Construction} Let $L=(L_1,L_2)$ be a smooth Hopf link such that $L_1$ is round with radius $r_1>0$. Let $D_1$ be the canonical flat disc bounded by $L_1$ and let $D_2$ be a smooth disc bounded by $L_2$ such that $D_2$ meets $L_1$ transversely in a single point. Suppose that $E_{h}\pitchfork D_2$ for some $h>0$, with intersection pattern $K\subset D_2$. Then there exists an isotopy $\Psi(s)\colon [0,1]\to \link(H)$, where $\Psi(s)=(\Psi_1(s),\Psi_2(s))$, satisfying:
\begin{enumerate}
    \item $\Psi(0)=(L_1,L_2)$ and $\Psi_1(s)=L_1$ for all $s\in [0,1]$.
    \item $\Psi_2(1)=\overline{L}_2$ meets $D_1$ transversely and only at its centrepoint.
\end{enumerate}
     The isotopy $\Psi(s)$ is a composition of two isotopies $\Phi(s)$ and $\Gamma(s)$. The first isotopy $\Phi(s)=(\Phi_1(s),\Phi_2(s))$ leaves $L_1$ unchanged, and successively removes intersections $\kappa\in K$ so that $\Phi_2(1)\cap B_{h}$ is a single unknotted segment running between $H_{h}^+$ and $H_{h}^-$. For each $\kappa\in K$, we assign a time $s(\kappa)\in[0,1]$ so that if $\Psi$ removes $\kappa_1$ before $\kappa_2$ then $s(\kappa_1)<s(\kappa_2)$. 
     Let $B_{h}$ be the convex ball bounded by $E_{h}$. The second isotopy  $\Gamma(s)=(\Gamma_1(s),\Gamma_2(s))$ also leaves $L_1$ unchanged and straightens the arc $\Phi_2(1)\cap B_{h}$ so that $\Gamma_2(1)\cap B_{h}$ agrees with the $z$-axis.    
\end{lem}

         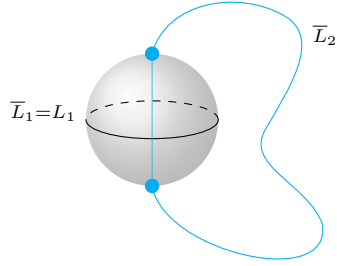
\begin{figure}[h]
    \begin{tikzpicture}[scale=2.5]
     \shade[ball color = gray!40, opacity = 0.4] (-.2,-.35) circle (0.35cm);
     \draw (-.55,-.35) arc (180:360:.35 and 0.1);
     \draw[dashed] (.15,-.35) arc (0:180:.35 and 0.1);
     \draw[cyan] (-.2,0) to[out=75,in=150](.5,.2) to[out=330,in=60](.4,-.4) to[out=240,in=110](.7,-.9) to[out=280,in=280](-.2,-.7);
     \filldraw [cyan] (-.2,0) circle (1pt);
     \filldraw [cyan] (-.2,-.7) circle (1pt);
     \draw[thin, cyan, opacity=.7] (-.2,0) to[](-.2,-.7);
     \node[below left] at (-.55,-.2) {$\scriptstyle \overline{L}_1=L_1$};
     \node[right] at (.6,.1){$\scriptstyle \overline{L}_2$};
    \end{tikzpicture}
    
    \caption{The Hopf link $(\overline{L}_1,\overline{L}_2)$ resulting from the isotopy $\Psi(s)$.}
    \label{fig:end of isotopy}
\end{figure}
\begin{proof}
We first describe an isotopy $\Phi(s)=(\Phi_1(s),\Phi_2(s))$ such that $\Phi_2(1)\cap B_{h}$ is a single unknotted segment running between $H_{h}^+$ and $H_{h}^-$. 

Label the points of intersection of $L_2$ with $E_h$ as $A_1,\ldots,A_n$ in cyclic order.  Each intersection point is decorated with $+$ or $-$ according to whether it represents an intersection with $H_{h}^+$ or $H_{h}^-$. The intersections of $D_2$ with $E_h$ are comprised of a collection of arcs running between pairs of the $A_j$ or circles, both lying in the interior of $D_2$. Since the linking number of $L_1$ and $L_2$ is 1, all arcs run between a pair of $A_j$ with the same sign, except for one. Let this special arc be called  $\alpha$. If $\alpha$ is the only arc, then we are done. Otherwise, we will remove all arc intersections inductively, until $\alpha$ is the only one remaining.

\begin{figure}[h]
    \centering
    \begin{tikzpicture}[scale=2.3]
     \draw (-.75,0) arc (180:360:.75 and 0.15);
     \filldraw[white] (0.025,-0.15) circle (0.3mm);
     \shade[ball color = red!40, opacity = 0.4] (0,0) ellipse (.75cm and .25cm);
     \draw[dashed] (.75,0) arc (0:180:.75 and 0.1);
        \draw[red] (-.58,.165) arc(140:269:.75 and .25);
        \draw[red] (.02,.-.25) arc(271:403:.75 and .25);
        \draw[red] (.53,.177) arc(45:83.5:.75 and .25);
        \draw[red] (.05,.25) arc(86:117:.75 and .25);
        \draw[red] (-.375,.22) arc(119:135:.75 and .25);
        \draw (.5,0) to[out=90,in=170](1,.4) to[out=350,in=0](.8,-.2) to[out=180,in=290](.65,-.15);
        \draw (.47,-.25) to[out=260,in=0](.1,-.9)to[out=180,in=355](-.1,.4)to[out=175,in=80](-.45,-.1);
        \draw (.45,.22) to[out=110,in=30](-.3,.5) to[out=210,in=80](-.6,0);
        \draw (-.46,-.21) to[out=255,in=0](-.57,-.5) to[out=180,in=260](-.64,-.15);
        \node at (-.85,0){$\scriptstyle L_1$};
        \node[below left, red] at (-.1,-.2){$\scriptstyle E_{h}$};
        \node[above right] at (1,.3){$\scriptstyle L_2$};
      
        \filldraw[gray!30!white!70!,draw=black] (2.5,0) circle (.75cm);

        \draw[red] (3.06,.5) to[out=190,in=170](3.22,-.2);
        \draw[red] (1.805,-.28) to[out=0,in=0](1.87,.4);
        \draw[red] (1.78,.2) to[out=0,in=10](1.76,-.1);
      
        \node at (3.13,.5){$\scriptstyle +$};
        \node at (3.28,-.2){$\scriptstyle -$};
        \node at (3.05,-.15) {$\scriptstyle \alpha$};
        
        \node at  (1.75,-.3){$\scriptstyle +$};
        \node at  (1.81,.4){$\scriptstyle +$};
        \node at  (2,-.25){$\scriptstyle \kappa_2$};
        
        \node at  (1.725,.2){$\scriptstyle -$};
        \node at  (1.705,-.1){$\scriptstyle -$};
        \node at  (1.70,.05){$\scriptstyle I$};
        \node at  (1.87,-.11){$\scriptstyle \kappa_1$};
        
        \node at (2.5,-.9){$\scriptstyle D_2$};

    \end{tikzpicture}
    \caption{On the left, the ellipsoid $E_{h}$ (in red) with equator $L_1$,  and $L_2$. On the right, a flat version of the disc $D_2$ bounded by $L_2$ is pictured with intersection pattern $K\subset D_2$ in red. We assume that $D_2$ meets $L_1$ transversely in a single point, so that any other intersections between $D_2$ and $E_h$ are circles (not pictured) lying in $E_h\setminus L_1$.}
    \label{fig:IntersectionPattern}
\end{figure}
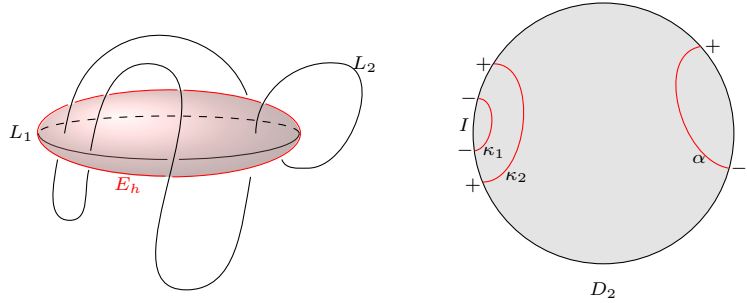

Each $\kappa\in K$ cobounds a disc $C(\kappa)$ on $D_2$ with an arc of $L_2$ between two of the $\{A_j\}$ on one side of $\alpha$.  If  $\lambda\subset E_h\cap D_2$ is a circle, then $\lambda$ bounds a disc in the interior of $D_2$. In this case $\lambda$ also bounds a disc on $E_{h}$. Since $D_2\cap L_2$ is a single point, it follows that this disc lies in $E_{h}\setminus L_1$. We call an arc $\kappa$ \emph{innermost} if there does not exist another arc $\kappa'$ such that $C(\kappa')\subset C(\kappa)$.  

For the inductive step, let $\kappa$ be an innermost arc with endpoints $A_{j}$ and $A_{j+1}$, and let $C=C(\kappa)$ be the disc $\kappa$ cobounds with the subarc $I=[A_j,A_{j+1}]\subset L_2$  between its endpoints. Since transversality is an open condition, there exists an $\varepsilon>0$ such that for any $\eta\in [-2\varepsilon,2\varepsilon]$,  $E_{h+\eta}$ intersects $D_2$ transversely in an intersection pattern $K(\eta)$ isotopic to $K$. In particular, there is a bijection between arcs in $K(\eta)$ and $K$. We choose $\varepsilon$ so that for all $\eta\in [-2\varepsilon,2\varepsilon]$, $\kappa(\eta)$ is disjoint from all other components in $(E_h\cap D_2)\setminus \kappa$. Moreover, for either all $\eta\in[0,2\varepsilon]$ or all $\eta\in [-2\varepsilon,0]$, the intersection $\kappa(\eta)$ isotopic to $\kappa$ bounds a disk $C(\eta)$ strictly containing $C$. Without loss of generality, assume this is true for all $\eta<0$, as the other case is symmetric. Let $\kappa'=\kappa(\varepsilon)$, $C'=C(\varepsilon)$ and $I'=C'\cap L_2$, then choose a collar\footnote{The isotopy defined here depends both on the choice of $\varepsilon$ and the choice of collar. In the parametrised version, we will choose a uniform $\varepsilon>0$ for all of $D^k$, and  a vary the choice of collar continuously with the parameter $t\in D^k$.} $\chi\colon I'\times [0,1]\rightarrow D_2$ such that
            \begin{enumerate}
                \item $C\subsetneq\chi\left(I'\times [0,\frac{1}{2}]\right)$,
                \item $\im(\chi)\subset C(2\varepsilon)$. 
            \end{enumerate}
From our choice of interval $\varepsilon$, (1) and (2) imply that $\im(\chi)$ is disjoint from all other intersections of $D_2$ with $E_{h}$ except possibly circles contained in $C$. 

Our method for removing $\kappa$ will be to push $L_2$ along $\chi$ past $\kappa$ itself. In the process, we may pass through circle intersections, but we can do this since these are disjoint from $L_1$, and since they are strictly contained in $C$, they will be completely removed by the end. Define an isotopy $\varphi_s\colon C(2\varepsilon)\rightarrow C(2\varepsilon)$  by pushing $L_2 $ along the transverse direction $\chi(\{x\}\times [0,1])$ so that $\varphi_1(C)\subset \chi(I\times[\frac{1}{2},1])$. This must be damped out near $\partial I'\times [0,1]$ to preserve smoothness of the image of $L_2$. By (1), $\varphi_1(D_2)\subset D_2$ no longer contains the intersection arc $\kappa$.

    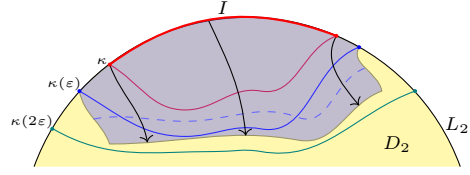
\begin{figure}[h]
        \centering
        \begin{tikzpicture}
            \filldraw[fill=yellow!40] (0,0) arc(160:20:3);
            \draw[red](1,1.35) to[out=-30,in=180] (2,.75) to[out=0,in=170](3,1) to[out=-10,in=200](4,1.73);

            \draw[blue](.6,1) to[out=-30,in=180] (2,.4) to[out=0,in=170](3,.5) to[out=-10,in=200](4.3,1.58);
        `   
            
            \draw[teal](.25,.5) to[out=-30,in=180] (2.4,.2) to[out=0,in=170](3,.2) to[out=-10,in=200](5.04,1);
            
            \filldraw[fill=blue!60,opacity=.4](.6,1) to[out=-80,in=110](1,.3) to[out=0,in=170](3.5,.35) to[out=-10,in=200](4.6,1) to[out=100,in=240](4.3,1.58) to[out=151,in=45.5](.6,1);

            \draw[blue!60,dashed](.785,.63) to[out=10,in=180](2,.6) to[out=0,in=160](3.5,.65) to[out=-20,in=200](4.37,1.3);

            \draw[->] (1,1.35) to[out=-70,in=100](1.5,.32);
            \draw[->] (2.3,1.95) to[out=-65,in=95](2.8,.41);
            \draw[->] (4,1.73) to[out=240,in=130](4.3,.81);

            \draw[thick,red] (4,1.73) to[out=159,in=36.4](1,1.35);
            \node at (2.5,2.1) {$\scriptstyle I$};
            
            \fill[red] (1,1.35) circle (.75pt);
            \fill[red] (4,1.73) circle (.75pt);
            \node at (.9,1.4) {\tiny$\scriptstyle\kappa$};

            \fill[blue] (.6,1) circle (.75pt);
            \fill[blue] (4.3,1.58) circle (.75pt);
            \node at (.4,1.1) {\tiny$\scriptstyle\kappa(\varepsilon)$};

            \fill[teal] (.24,.5) circle (.75pt);
            \fill[teal] (5.04,1) circle (.75pt);
            \node at (-.03,.6) {\tiny$\scriptstyle\kappa(2\varepsilon)$};

            \node at (4.8,.3) {$\scriptstyle D_2$};
            \node at (5.6,.55) {$\scriptstyle L_2$};

        \end{tikzpicture}
        \caption{The collar used to isotope away an innermost arc intersection. The arc $\kappa$ and the bounding interval $I\subset L_2$ is shown in red, while the pushoffs $\kappa(\varepsilon)$ and $\kappa(2\varepsilon)$ are shown in blue and teal, respectively. The interval $I'$ is the segment between the two blue endpoints on the semicircle. A collar $\chi$ of $I'$ is shown in pale blue, and the dashed blue line represents the boundary of $\chi(I'\times[0,\frac{1}{2}])$, which contains $C$. The isotopy pushes $C$ across the collar along the transverse black arrows.}
        \label{fig:enter-label}
    \end{figure}

    Choose a total order on the arc intersections (not including $\alpha$) $A=\{\kappa_1,\ldots, \kappa_m\}$ such that if $C(\kappa_i)\subset C(\kappa_j)$ then $i<j$. To each $\kappa_i$ we assign a time $s(\kappa_i)\in [0,1]$ such that if $i<j$ then  $s(\kappa_i)<s(\kappa_j)$. Pick a $\delta$ that is smaller than the max difference between any two $s$ values, and define $\Phi(u)$ to do the isotopy for intersection $\kappa_i$ in the time interval $[s(\kappa_i), s(\kappa_i)+\delta]$. Since $\Phi(s)$ only affects points in $D_2$ away from $\alpha$, $L_1$ is unchanged by $\Phi(s)$.
    
    Let $\Phi(1)=(\Phi_1(1),\Phi_2(1))$ be the Hopf link resulting from this isotopy. Then $\Phi_1(1)=L_1$ and  $\Phi_2(1)\cap B_{h}=\alpha$ may be regarded as an embedded arc into $B_{h}$. We necessarily have that one end of this arc lies on $H_h^+$ and one end on $H_h^-$. By Lemma \ref{lem: NP SP arc space contractible} the space of such embeddings is contractible, in particular it is connected, so there is an isotopy $\Gamma(s)=(\Gamma_1(s),\Gamma_2(s))$ that straightens the arc so that it agrees with the portion of the $z$-axis between NP and SP. We can choose the isotopy so that it does not change $L_1$, nor the portion of $\Phi_2(1)$ outside a neighbourhood of $B_{h}$. Concatenating $\Phi(s)$ and $\Gamma(s)$ is the desired isotopy $\Psi(s)$, satisfying (1) and (2).
\end{proof}

\subsection{Upgrading to a parametrised isotopy}
We now have all the ingredients to prove Theorem \ref{abcthm: round space equivalence}, via a parametrised version of the isotopy in Section~\ref{sec:Isotopy-Removing-Intersections}.

\begin{proof}[Proof of Theorem \ref{abcthm: round space equivalence}]
    Let $f\colon (D^k,\partial D^k) \to (\link(H),\rlink(H))$ represent an arbitrary element of $\pi_k(\link(H),\rlink(H))$. Our aim is to homotope $f$ rel $\partial D^k$ so that its image lies entirely in $\rlink(H)$. Let $L_1(t)$ and $L_2(t)$ denote the $D^k$-families of embeddings obtained by restricting to $L_1$ and $L_2$, respectively. By Lemma \ref{lem:One-Component-Round}, we may assume that $L_1(t)$ is round for all $t\in D^k$. Recall that round unknots bound canonical flat round discs in $\R^3$. In particular, for every $t\in D^k$, $L_1(t)$ bounds a canonical round disk $D_1(t)$. Since $L_2(t_0)$ is also round, we let $D_2(t_0)$ be the canonical flat disc, and obtain a family of discs $D_2(t)$ by isotopy extension. Therefore, since $D_2(t_0)\cap L_1(t_0)$ is a single point, $D_2(t)\cap L_1(t)$ is a single point for all $t\in \partial D^k$. 
    
    After this initial set-up, the proof of the theorem proceeds via the following three steps. 
    \begin{itemize}
        \item[Step 1] For each $t\in D^k$ we define an isotopy $\Psi[t](s)\colon [0,1]\to \link(H)$ such that $\Psi[t](0)=f(t)$. The notation $[t]$ indicates that the choice of isotopy depends on $t$ but is not (yet) continuous in $t$. We write $\Psi[t](s)=(\Psi_1[t](s),\Psi_2[t](s))$ for the two components of the Hopf link thus obtained, and set $\Psi[t](1):=(\overline{L}_1[t],\overline{L}_2[t])$.
        We construct $\Psi[t](s)$ so that $\Psi_1[t](s)=L_1(t)$ for all $s\in[0,1]$ and so that $\overline{L}_2[t]$ agrees with a segment of the line perpendicular to $D_1(t)$ through its centrepoint in a neighbourhood of $D_1(t).$ 
        \item[Step 2] We show that the isotopies $\Psi[t](s)$ can be defined continuously across $D^k$ to yield an isotopy $
        \Psi(t,s)\colon D^k \times [0,1]\to \link(H)$.
        \item[Step 3] We use the intersection properties of $\overline{L}_2(t)$ with $D_1(t)$ to round $L_2(t)$ continuously across $D^k$, while keeping $L_1(t)$ unchanged, and hence  round.
    \end{itemize}

    Before we start with Step 1, we define a cover of $D^k$ as follows. As in Section \ref{sec:Isotopy-Removing-Intersections}, we consider a family of concentric ellipsoids $E_h(t)$ of varying height $h \in (0,\infty)$ with equator $L_1(t)$ for all $t\in D^k$. For each $t\in D^k$ the disc $D_2(t)$ is transverse to some ellipsoid $E_h(t), h \in (0,\infty)$.  Note that the parameter $h$ defining the ellipsoid is the distance from the poles to the plane containing $D_1$, and hence the family of ellipsoids $E_h(t)$ varies continuously with $L_1(t)$.  Transversality is an open condition, so by compactness of $D^k$ we may cover $D^k$ with a finite collection of path-connected open sets $\{U_i\}_{i\in I}$ such that there exists $h_i \in (0,\infty)$ with $L_2(t)$ transverse to the ellipsoid $E_{h_i}(t)$ whenever $t\in U_i$. 
    
    The specific value of $h_i$ does not matter for this proof, so we denote $E_{h_i}(t)$ by $E_i(t)$. Let $B_i(t)$ be the convex ball with boundary $E_i(t)$ whenever $t\in U_i$. We also obtain a continuous family of upper and lower hemispheres  across $U_i$, denoted $H_i^+(t)$ and $H_i^-(t)$, so $E_i(t)=H_i^+(t)\cup_{L_1(t)}H_i^-(t)$. The north pole (NP) is the centrepoint of $H_i^+(t)$ and the south pole (SP) is the centrepoint of $H_i^-(t)$.

    \textbf{Step 1} 

    For $t\in U_i$, let $K_i(t)\subset E_i(t)\cap D_2(t)$ be the intersection pattern.  We apply Lemma \ref{lem:Isotopy-Construction} to obtain an isotopy $\Psi[t](s)_i$ with the claimed properties. Because there are only finitely many $U_i$, we can choose one $\varepsilon>0$ such that for all $U_i$ and for all $h\in[-2\varepsilon,2\varepsilon]$, we have that $E_{h_i+h}(t)\pitchfork D_2(t)$.    Moreover, since $U_i$ is path-connected and $D_2(t)\pitchfork E_i(t)$ on $U_i$, we obtain an order-preserving bijection between arcs of $K_i(t)$ and $K_i(t')$ for $t,t'\in U_i$, therefore we may choose the function $s_i\colon K_i(t)\to [0,1]$ and time interval $\delta_i$ from Lemma~\ref{lem:Isotopy-Construction} uniformly for $t\in U_i$.

    Since a given $t$ may lie in more than one $U_i$, we define $\Psi[t](s)$ by `gluing' the isotopies $\{\Psi[t](s)_i\}_{\{i | t\in U_i\}}$ together. Pick functions $s_i$ and $\delta_i$ above so that the set of isotopies $\{\Psi[t](s)_i\}_{\{i | t\in U_i\}}$ are supported at different times. Furthermore if $t\in U_i\cap U_j$, then the intersections of $L_2(t)$ with $E_i(t)$ and $E_j(t)$ are comparable/nested in $D_2(t)$, so we can arrange for the $s(\kappa)$ values for the intervals to appear in such an order that the innermost intersections are isotoped away first. It follows from these choices that the set of isotopies $\{\Psi[t](s)_i\}_{\{i | t\in U_i\}}$ can be performed simultaneously.
        
    Let $\{U_i'\}$ be a cover subordinate to $\{U_i\}$. Define a partition of unity $\varphi_i(t)\colon D^k \to [0,1]$ such that $\varphi_i(t)=1$ for $t\in U_i'$, and $\varphi_i(t)=0$ for $t\notin U_i$. Denote by $\Psi[t](s)_i^\varphi$ the isotopy which is equal to $\Psi[t](s)_i$ for $s \in [0,\varphi_i(t)]$, and the identity otherwise. Let $\Psi[t](s)$ be the isotopy which simultaneously does the isotopies 
        $\{\Psi[t](s)_i^{\varphi}\}_{\{i | t\in U_i\}}$. 

        \textbf{Step 2} 
        
        We show that $\Psi[t](s)$ can be defined continuously across $D^k$, to yield an isotopy $\Psi(t,s)\colon D^k \times [0,1]\to \link(H)$.
        Step 1 gave us an isotopy $\Psi[t](s)\colon [0,1] \to \link(H)$ for $t\in D^k$, such that if $t\in U_i'$ then at the end of the isotopy $\overline{L}_2(t)$ transversely intersects $B_{i}(t)$ in a straight arc running from NP to SP. Pick a simplicial subdivision of $D^k$, such that each simplex lies in some $U_i'$. For $t$ in the set of 0-simplices, let $\Psi(t,s)=\Psi[t](s)$. For the inductive step we show that the isotopy can be continuously extended from $j$- to $(j+1)$-simplices for $0\leq j\leq k-1$. We show this for the two isotopies $\Phi(s)$ and $\Gamma(s)$ comprising $\Psi(s)$ in turn. 
        \begin{enumerate}[(a)]
            \item The isotopy $\Phi(s)$ which removes an intersection $\kappa\in K(t)$ is defined using a collar. Since a $(j+1)$-simplex is entirely in one $U_i'$ the intersection pattern $K(t)$ varies only by isotopy across the simplex. The space of collars for intervals in $\partial D_2(t)$ is contractible by Lemma~\ref{lem:Boundary-Collar-Contractible}. Thus, we can always extend the choice of collar over the $(j+1)$-simplex.
            \item At $t\in D^k$ we straighten the unknotted arc $L_2(t)\cap B_i$ to the straight line from NP to SP. Over the simplex, which is contained in $U_i'$, this unknotted arc will have the property that one boundary is in $H^+_i$ and one boundary in $H^-_i$, since the equator $L_1(t)$ is disjoint from $L_2(t)$. By Lemma~\ref{lem: NP SP arc space contractible} the space of such embeddings is contractible, so we can extend the straightening over the $(j+1)$-simplex. (Note that the property of intersecting $B_i$ in a straight line through the poles passes to $B_j$ when $j< i$, and so can be satisfied for multiple $U_i$ at once.)
        \end{enumerate}

        Note that during this isotopy we must ensure that the image of $\Psi(t,s)$ remains in $\rlink(H)$ when $t\in S^{k-1}$. We can do this by first homotoping $f$ such that it agrees with $f|_{S^{k-1}}$ radially on a collar $S^{k-1}\times [0,\epsilon]$ of the boundary. Now we perform the homotopy $\Psi(t,s)$ on the disc $D^k\setminus (S^{k-1}\times [0, \epsilon))$, and extend it across the collar so that on the boundary $S^{k-1}\times \{0\}$ the homotopy restricts to the identity.
        
    \textbf{Step 3}
        At the end of Step 2 we have an isotopy  $\Psi(t,u)$ defined continuously over $D^k$, such that $\Psi(t,0)=f(t)$, and after applying  isotopy extension, at $s=1$ the image of $L_1(t)$ is round, and the image of $L_2(t)$ intersects the disk $D_2(t)$ once only, at the centrepoint of $D_2(t)$ (see Figure \ref{fig:end of isotopy}). We now shrink $L_1(t)$ continuously over the whole $D^k$, using the round disc $D_1(t)$ to guide the isotopy. We do this until $L_1(t)$  has radius $\rho$ for some very small $\rho>0$. Note that during this isotopy the image of the embedding on $\partial D^k$ remains a round Hopf link. Now we consider our resultant $f$ over $D^k$, and use Lemma~\ref{lem: rounding one unknot} to homotope $f|_{L_2}$ such that $L_2(t)$ is round for all $t\in D^k$. We extend this to $L_1(t)$ by isotopy extension, but note that since $L_1(t)$ can be made arbitrarily small, we can choose it small enough so that it is unaffected by the isotopy.
\end{proof}

\section{Computing the homotopy type}\label{sec: homotopy type}

In this section, we prove our main result: that $\rlink(H)$ is homotopy equivalent to the prism manifold formed as the quotient of $S^3$ by the quaternion group $\Q_8$. First consider the space $\rlink(U)$ consisting of embeddings of a single round unknot $U$ in $\R^3$. Each $\rho \in \rlink(U)$ is determined by three parameters $(p,r,\ell)$ where:

\begin{itemize}
    \item $p\in \R^3$ is the center of the canonical disk of $\rho$,
    \item $r>0$ is the radius of the canonical disk of $\rho$,
    \item $\ell\in \R P^2$ denotes the line normal to the the canonical disk of $\rho$.

\end{itemize}
Therefore $\rlink(U)$ is homeomorphic to $\R^3\times \R_+\times \R P^2$, a 6-dimensional manifold homotopy equivalent to $\R P^2$. Note that if we orient the image of $\rho$, this determines an orientation on the canonical disk it bounds, and thus picks out a choice of unit normal vector. Thus, the space $\rlink(U)^+$ of round embeddings of single oriented unknot  in $\R^3$ is homeomorphic to $\R^3\times \R_+\times S^2$.

For any embedding $\rho\in \rlink(H)$, there is a neighborhood where we can vary the embeddings of each component independently.  Indeed, since each component is compact, there is a minimum distance between them which varies continuously with $\rho$. Thus every round embedding of each component near $\rho$ will still yeild a Hopf link. It follows that $\rlink(H)$ is locally homeomorphic to an open subset of the product space $\rlink(U)\times\rlink(U)$,  hence a non-compact 12-manifold. 

Let $\remb(H,\R^3)\subset \emb(H,\R^3)$ denote the subspace of round embeddings equipped with a unit speed parametrisation of each circle.  Reparametrising gives a free action of the compact Lie group $K=\Isom(S^1\sqcup S^1)=\big(\Or(2)\times \Or(2)\big)\rtimes \Z/2$ on  $\remb(H,\R^3)$, where $\Or(2)$ acts as the isometry group of each circle, and $\Z/2$ exchanges the two circles. We can identify the quotient $\remb(H,\R^3)/K$ with $\rlink(H)$. 

Let $\prlink(H)^\pm$ be the intermediate quotient  $\remb(H,\R^3)/(\SO(2)\times \SO(2))$. 
Since $\SO(2)\times \SO(2)$ is the identity component of $(\Or(2)\times \Or(2))\rtimes \Z/2$, the map $\prlink(H)^\pm\rightarrow \rlink(H)$ is an 8-fold covering with deck group $K/(\SO(2)\times\SO(2))\cong(\Z/2\times \Z/2)\rtimes\Z/2$. Thus $\prlink(H)^\pm$ is the space of oriented, labelled, round, unparametrised embeddings of a Hopf link. Here the ``$\mathcal{P}$'' stands for ``pure", since we may not exchange the two components. Depending on the orientations of the two components, the oriented linking number is either $+1$ or $-1$. Since the oriented linking number is an isotopy invariant, this means $\prlink(H)^\pm$ has at least two connected components. Let $\prlink(H)^+\subset \prlink(H)^\pm$ be the subspace of those embeddings with  $+1$ oriented linking number.

\begin{lem}\label{lem:8FoldCover}
$\prlink(H)^+$ is a connected, 4-fold covering of $\rlink(H)$  with deck group $G=\Z/2\times \Z/2$. The first $\Z/2$-factor reverses the orientations of both link components, while the second  exchanges the labels of the components.
\end{lem}
\begin{proof}
We know that the deck group of $\prlink(H)^\pm\rightarrow \rlink(H)$ is isomorphic to $(\Z/2\times\Z/2)\rtimes \Z/2$ Let $H=C_1\sqcup C_2$ be the round embedding described in Example \ref{example:Basic Hopf}. Let $\alpha_i$ reverse the orientation of $C_i$ and let $s$ swap $C_1$ and $C_2$. Each $\alpha_i$ changes the sign of the oriented linking number, hence exchanges $\prlink(H)^+$ and $\prlink(H)^-$. However, their product $\alpha=\alpha_1\alpha_2$ changes both orientations simultaneously and can be achieved by a path of round embeddings in $\prlink(H)^+$ starting at $H$ that flips over both $C_1$ and $C_2$ (rotation by $\pi$ about the $x$-axis). Similarly, the automorphism $s$ which exchanges $C_1$ and $C_2$ may be obtained by path of round embeddings (a rotation about $(1/2,0,0)$, the midpoint of the arc of intersection). A schematic of these rotations is depicted in Figure~\ref{fig:rotations}.

Since $\rlink(H)$ is connected by Theorem \ref{abcthm: round space equivalence}, this implies  $\prlink(H)^\pm$ has exactly two connected components, which are exchanged by the action of either $\alpha_i$.  On the other hand, $\alpha$ and $s$ commute, and generate a normal subgroup of the deck group isomorphic to $\Z/2\times \Z/2$. 
In particular, $\prlink(H)^+$ is a 4-fold covering of $\rlink(H)$ with deck group $G=\Z/2\times \Z/2$, generated by $\{\alpha,s\}$. 
\end{proof}

\begin{figure}
        \centering
        \begin{tikzpicture}
            \draw (1,0) ellipse (1 and 1);                  
            \filldraw[white](0.2,-0.6) circle (2pt);
            \draw (0,0) ellipse (1 and 0.6);                
            \filldraw[white](0.2,0.6) circle (2pt);
            \draw (1,1) arc[start angle=90, end angle=180, radius=1];
            \draw[->] (3,0.15) arc[start angle=60, end angle=240, radius=.2];
            \node at (3.3,0) {$\alpha$};
            \draw[gray] (-2,0)--(3,0);             
            \draw[thick,->,red, dotted] (-2,0)--(3,0); 
            \draw[->,gray] (0,-2)--(0,2);
            \draw[->,gray] (-1.8,0.6)--(2.5,-0.8);
            \filldraw(0,0) circle (.5pt);
            \filldraw(1,0) circle (.5pt);      
            \draw[thick,->,red, dotted] (-.3,-1.8)--(1.3,1.8);           
            \draw[->] (1.05,1.7) arc[start angle=140, end angle=350, radius=.2];
            \node at (1.6,1.8) {$s$};          
            \draw[thick,blue] (0,0)--(1,0);
            \filldraw[blue] (0.5,0) circle (.5pt); 
        \end{tikzpicture}
        \caption{The basic model $H$, and the axes of rotation for the two generators $\alpha$ and $s$.}
        \label{fig:rotations}
    \end{figure}
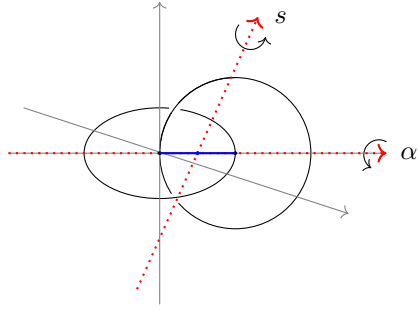

We therefore have that $\prlink(H)^+$ is a connected component of $\prlink(H)^\pm$ corresponding to linking number $+1.$ For any embedding $\rho\in \prlink(H)^+$, the link components are oriented and carry distinct labels $C_1,~C_2$.

\begin{lem}\label{lem:PRHL Homotopy type}
There exists a $G$-equivariant deformation retraction of $\prlink(H)^+$ onto a subspace $Y$ homeomorphic to $\SO(3)$.
\end{lem}
\begin{proof}

Given $\rho\in \prlink(H)^+$, let $p_i=p_i(\rho)$ be the center point of $C_i$, and let $r_i=r_i(\rho)$ be the radius of $C_i$ for $i=1,2$. Since $C_i$ is oriented, it has a well-defined unit normal vector $n_i$. Finally, let $m$ be the midpoint of the arc of intersection and let $q_i$ be the endpoint of the arc of intersection $\arc(\rho)$ lying in the disk bounded by $C_i$. Observe that $\prlink(H)^+$ has a free action by $\Isom^+(\R^3)=\R^3\rtimes \SO(3)$. Here the $\R^3$ factor acts by translation and $\SO(3)$ acts by rotations. 

We will construct a sequence of deformation retractions from $\prlink(H)^+$ onto a subspace $Y$, which we will show is homeomorphic to $\SO(3).$ At each stage, we will show that the deformation retraction is $G$-equivariant.

\begin{itemize}
    \item Let $X_0\subset \prlink(H)^+$ denote the subspace where $m=(1/2,0,0)$.  Using the free translation action of $\R^3\leq \Isom^+(\R^3)$, we see that $\prlink(H)^+$ deformation retracts onto $X_0$. Since the action of $G$ does not affect the image of the embedding, it acts trivially on $m$, hence this deformation retraction is $G$-equivariant.
    \item Let $\theta$ be the angle between $n_1$ and $n_2$. This is the dihedral angle made between the two planes containing $C_1$ and $C_2$, computed via the formula $\langle n_1,n_2 \rangle =\cos(\theta)$. By rotating $C_1$ and $C_2$ simultaneously about the arc of intersection, we can vary the dihedral angle to be any value in the open interval $(0,\pi)$ while leaving the arc of intersection fixed. Define a deformation retraction from $X_0$ onto a subspace $X_1$ where the dihedral angle of each embedding is $\pi/2.$ The fiber of this deformation retraction is an open interval $(0,\pi)$  hence homeomorphic to $\R$. Since $\alpha$ negates both $n_1$ and $n_2$ while $s$ switches $n_1$ and $n_2$, neither changes $\theta$. Therefore the deformation retraction $X_0\rightarrow X_1$ is $G$-equivariant.
    \item Next consider the function $f\colon X_1\rightarrow \R$ defined by $f(\rho)=r_1(\rho)-r_2(\rho)$.  The subspace $X_2=f^{-1}(0)$ consists of those embedding where the radii of $C_1$ and $C_2$ are equal.  We now define a strong deformation retraction onto $X_2$ by increasing the radius of the smaller of the two circles until $r_1=r_2$, while leaving the arc of intersection fixed. 
    In this case the fiber over a point in $X_2$ is homeomorphic to $\R.$ Observe that $\alpha$ acts trivially on $(r_1,r_2)$ while $s$ permutes the two coordinates. Since the deformation retraction of $X_1$ onto $X_2$ increases the $\min\{r_1,r_2\}$, it depends only on the set $\{r_1,r_2\}$, hence is equivariant with respect to the action of $G$.

    \item Now $C_1$ and $C_2$ have a common radius $r=r_1=r_2$. We rescale the common radius $r$ to be 1, again leaving the arc of intersection fixed. This is a deformation retraction onto a subspace $X_3\subset X_2$. Since  the common radius $r$ can be any positive number, the fiber of the retraction from $X_2$ to $X_3$ is homeomorphic to $ \R_+\cong \R$. As we observed above, $G$ acts trivially on $(r_1,r_2)$ when they are equal. Therefore the retraction $X_2$ to $X_3$ is also $G$-equivariant.
    
    \item For the last step, we arrange that the endpoints of $\arc(\rho)$ are $p_1$ and $p_2$. The diameter of $C_2$ that is parallel to $\arc(\rho)$ meets $C_2$ at two points. Let $y_1$ be the endpoint of this diameter which is closest to $p_1$. (If $V$ is the hyperplane passing through $p_2$ which is orthogonal to the line containing $\arc(\rho)$, then $y_1$ is the endpoint on the same side of $V$ as $\arc(\rho)$). Similarly, the diameter of $C_1$ parallel to $\arc(\rho)$ meets $C_1$ at two points and we set $y_2$ to be the endpoint closest to $p_2$. Let $\gamma_i$ be the vector from  $p_i$ to $y_i$. By construction, the segment from $p_1$ to $y_2$ is parallel to the segment from $p_2$ to $y_1$ and has the same length, namely the common radius $r=1$.  Since $\gamma_i$ translates $p_i$ to $y_i$,  it follows that
    $\gamma_1=-\gamma_2$. A schematic is shown in Figure~\ref{fig:arc int homotopy}.

    For $t\in [0,1]$, we simultaneously translate $C_1$ along $\frac{t}{2}\gamma_1$ and $C_2$ along $\frac{t}{2}\gamma_2$ so that when $t=1$, both have been translated distance $\frac{1}{2}$. After the motion, the image of $p_1$ and $p_2$ are at the endpoints of the same segment of length $r=1$ parallel to $\arc(\rho)$, hence $\arc(\rho)$ is the straightline segment from $p_1$ to $p_2$. Let $X_4$ be the subspace of $X_3$ where the endpoints of the arc of intersection are $p_1$ and $p_2$.  Since these translations do not changed the radii, the midpoint of $\arc(\rho)$ nor the angle $\theta$, we have thus defined a deformation retraction from $X_3$ onto $X_4$. The action of $\alpha$ leaves both $\gamma_i$ invariant, while $s$ exchanges the $p_i$ and the $y_i$, hence exchanges $\gamma_1$ with $\gamma_2$. Thus, the deformation retraction is $G$-equivariant.

    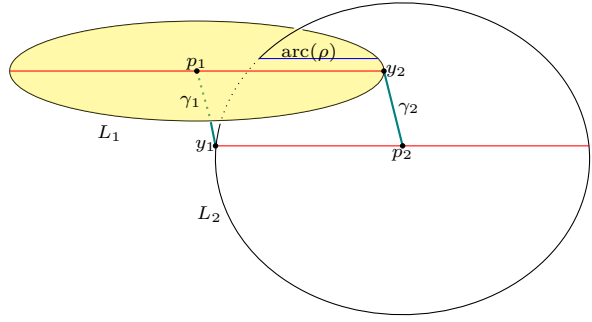
\begin{figure}
        \centering
        \begin{tikzpicture}[scale=1.65]
            \draw[thick,teal](4,0)--(4.15,-.6);
            \draw[thick,teal,dotted](2.5,0)--(2.612,-.4);
            \draw[thick,teal](2.65,-.6)--(2.612,-.41);
            \draw(2.5,0) ellipse (1.5 and 0.4);
            \filldraw[white](3.31,0.32) circle (1pt);
            \filldraw[fill=yellow!90, fill opacity=.4, draw=none](2.5,0) ellipse (1.5 and 0.4);
            \draw(3,.1) arc (140:-193:1.5 and 1.25);
            \draw[dotted](3,.1) arc(140:165:1.5 and 1.25);
            \draw[red] (1,0)--(4,0);
            \draw[red] (2.65,-.6)--(5.65,-.6);
            \draw[blue] (3,.1)--(3.95,.1);
            
            \filldraw(2.5,0) circle (.5pt);
            \filldraw(4.15,-.6) circle (.5pt);
            \filldraw(2.65,-.6) circle (.5pt);
            \filldraw(4,0) circle (.5pt);
            \node at (2.5,.06) { $\scriptstyle p_1$};
            \node at (4.15,-.67) { $\scriptstyle p_2$};
            \node at (4.1,0) { $\scriptstyle y_2$};
            \node at (2.57,-.6) { $\scriptstyle y_1$};
            \node at (4.2,-.3) { $\scriptstyle \gamma_2$};
            \node at (2.45,-.25) { $\scriptstyle \gamma_1$};
            \node at (3.4,.15) { $\scriptstyle \arc(\rho)$};
            \node at (1.8,-.5) {$\scriptstyle L_1$};
            \node at (2.6,-1.15) {$\scriptstyle L_2$};
        \end{tikzpicture}
        \caption{A schematic of the last step of the deformation retraction, from $X_3$ to $X_4$. $L_1$ and $L_2$ both have radius 1 and meet along a right angle. We move $L_1$ along $\gamma_1$ and simultaneously move $L_2$ along $\gamma_2$ in the opposite direction until the endpoints of the arc of intersection $\arc(\rho)$ coincide with $p_1$ and $p_2$.}
        \label{fig:arc int homotopy}
    \end{figure}
    
\end{itemize}
Set $Y=X_4\subset \prlink(H)^+$. Then $Y$ is the subspace of embeddings satisfying
\begin{enumerate}
    \item $m=(1/2,0,0)$,
    \item $\theta=\angle(n_1,n_2)=\pi/2$,
    \item $r_1=r_2=1$,
    \item $p_1$ and $p_2$ are the endpoints of the arc of intersection.
\end{enumerate}
In particular, the arc of intersection is a unit length line segment with midpoint equal to $(1/2,0,0)$.  Setting $v=p_2-p_1$, every point of $Y$ is completely specified by the the pair of unit vectors $(n_1,v)$ since the positive linking condition implies that $n_2$ must be the cross product $n_1\times v$. The triple $(n_1,v,n_2)$ therefore determines a unique frame in $\SO(3).$  Thus $Y$ is homeomorphic to $\SO(3)$ and there is a $G$-equivariant deformation retraction of $\prlink(H)^+$ onto $Y$.
\end{proof}

\begin{rem}
In fact, the above proof shows that $\prlink(H)^+\cong \R^9\times \SO(3)$, and that there exists a $G$-equivariant projection onto the $\SO(3)$-factor.
\end{rem}

Recall that the universal covering group of $\SO(3)$ is $\SU(2)$, the group of unit quaternions, and thus is homeomorphic to $S^3$. Multiplication by unit quaternions is length-preserving with respect to the standard Euclidean metric on $\R^4$ and only fixes the origin. It therefore acts by orientation-preserving isometries on the standard round metric on $S^3$. We thus we obtain a homomorphism $\SU(2)\rightarrow \Isom^+(S^3)=\SO(4)$ where each $q\in \SU(2)$ acts by conjugation $c_q\colon x\mapsto qxq^{-1}$. The kernel of this map is the center of $\SU(2)$, namely $\{\pm 1\}$. On the other hand, this action clearly fixes a point in $S^3$ (corresponding to $1\in \SU(2)$) hence its image lies in $\SO(3)\leq \SO(4)$. For dimension reasons the map $\SU(2)\rightarrow \SO(3)$ must be onto, giving rise to a central extension
\begin{equation}\label{eqn:S3cover}1\to \Z/2\to \SU(2)\to \SO(3)\to 1\end{equation}
where $\Z/2\cong \{\pm 1\}\leq \SU(2)$. 

The \emph{quaternion group} $\Q_8$ is the multiplicative group consisting of the 8 quaternions $\{\pm 1,\pm i,\pm j,\pm k\}\leq \SU(2)$, with relations 
\[i^2=j^2=k^2=-1, ~ij=k,~jk=i,~ki=j\]
The image of $\Q_8$ in $\SO(3)$ is the subgroup isomorphic to $\Z/2\times \Z/2$, generated by rotations by $\pi$ in $\R^3$ about each of the three coordinate axes.  In terms of the conjugation action described above, the nontrivial elements are represented by the matrices:\[\overline{\i}=\left(\begin{array}{ccc}1&0&0\\0&-1&0\\0&0&-1\end{array}\right),~\overline{\j}=\left(\begin{array}{ccc}-1&0&0\\0&1&0\\0&0&-1\end{array}\right),~\overline{\text{k}}=\left(\begin{array}{ccc}-1&0&0\\0&-1&0\\0&0&1\end{array}\right).\]

Thus, by Equation \ref{eqn:S3cover}, $\Q_8$ may also be regarded as a central extension  \begin{equation}\label{eqn: Quaternion}
    1\rightarrow \Z/2\rightarrow \Q_8\rightarrow \Z/2\times \Z/2\rightarrow 1
\end{equation}
where the pre-image of each $\Z/2$ subgroup of $\Z/2\times \Z/2$ is isomorphic to $\Z/4$. Combining Equations \ref{eqn:S3cover} and \ref{eqn: Quaternion}, we obtain that $\SU(2)/\Q_8\cong\SO(3)/\left(\Z/2\times \Z/2\right)$ where $\Z/2\times \Z/2=\left\langle \overline{\i},\overline{\j}, \overline{\text{k}}\right\rangle$.

\begin{restate}{Theorem}{abcthm: homotopy type}
The embedding space $\link(H)$ is homotopy equivalent  to $S^3/\Q_8.$
\end{restate}
\begin{proof}By Theorem \ref{abcthm: round space equivalence}, it is enough to show that $\rlink(H)\simeq S^3/\Q_8$. By Lemma \ref{lem:8FoldCover}, we know that $\prlink(H)^+$ is 4-fold cover of $\rlink(H)$ with deck group $G=\Z/2\times \Z/2$, which acts by reversing the orientation of both components and exchanging the two components. Recall that  $\alpha$ is the element that reverses the orientations of both components, while $s$ is the generator that switches $C_1$ and $C_2$.

By Lemma \ref{lem:PRHL Homotopy type}, we know that $\prlink(H)^+$ is homotopy equivalent to $\SO(3)$, realised as the subspace of round embeddings $Y$ satisfying:
\begin{itemize}
\item the midpoint of the arc of intersection is $m=(1/2,0,0)$;
\item the planes containing $C_1$ and $C_2$ are orthogonal;
\item $C_1$ and $C_2$ have radius 1 and linking number $+1$;
\item the centers $p_1$ and $p_2$ of $C_1$ and $C_2$, respectively, are the endpoints of the arc of the intersection.
\end{itemize}
Each element of $Y$ defines a unique frame in $\SO(3)$ as the ordered triple of unit vectors $(v,n_1, n_2)$, where $v=p_2-p_1$ and $n_i$ is the normal vector to the plane containing $C_i$. The action of the deck group preserves $Y$, since it changes the orientations of the $C_i$ and exchanges the labels. Since the homotopy equivalence from $\prlink(H)^+\rightarrow Y$ is equivariant to with respect to the deck group $G$, to prove the corollary, it suffices to show that $Y/G\cong \SU(2)/\Q_8$.

Starting from the basepoint $H$, the element $\alpha$ can be realized by a rotation by $\pi$ about the $x$-axis. The generator $s$ can be realized by a rotation by $\pi/2$ about the $x$-axis followed by a rotation of $\pi$ about the axis parallel to the $z$-axis and passing through $(1/2,0,0).$ As elements of $\SO(3)$, these two rotations are represented by the matrices \[\alpha=\left(\begin{array}{ccc}1&0&0\\0&-1&0\\0&0&-1\end{array}\right),~s=\left(\begin{array}{ccc}-1&0&0\\0&0&1\\0&1&0\end{array}\right).\]
 Since $\alpha,s$ are commuting orthogonal matrices, they are simultaneously diagonalisable, and $\alpha$, $s$ and the product $\alpha\cdot s$ diagonalise to the three matrices
\[\overline{\i}=\left(\begin{array}{ccc}1&0&0\\0&-1&0\\0&0&-1\end{array}\right),~\overline{\j}=\left(\begin{array}{ccc}-1&0&0\\0&1&0\\0&0&-1\end{array}\right),~\overline{\text{k}}=\left(\begin{array}{ccc}-1&0&0\\0&-1&0\\0&0&1\end{array}\right).\]
which are precisely the images of $i$, $j$, and $k$ under the universal covering map $\SU(2)\rightarrow \SO(3).$ Therefore $Y/G$ is homeomorphic to $\SO(3)/\left\langle \overline{\i},\overline{\j}, \overline{\text{k}}\right\rangle\cong \SU(2)/\Q_8$.
\end{proof}

\section{Hopf links in $S^3$}\label{sec: great hopf links}

In this section we briefly discuss the translation of the above results to embeddings of Hopf links in $S^3$. Let $\link(H,S^3)$ denote the space of smooth, unparametrised embeddings of a Hopf link in $S^3$. In this case our basepoint is the embedding  consisting of the two great circles \[H=\{(x,y,0,0)\in \R^4\mid x^2+y^2=1\}\cup \{(0,0,z,w)\in \R^4\mid z^2+w^2=1\}.\]

\begin{defn}
Recall that a great circle of $S^3$ is the  intersection of $S^3$ with a 2-plane in $\R^4$. 
Let $\rlink(H,S^3)$ be the subspace of $\link(H,S^3)$ where each component is a great circle, and the two components lie in orthogonal subspaces of $\R^4$. We will refer to a point of $\rlink(H,S^3)$ as a \emph{great Hopf link}.
\end{defn}

The Smale conjecture (see (7) in \cite[Appendix]{Hatcher}) is also equivalent to the assertion that the unparametrised embedding space of an unknot in $S^3$ is homotopy equivalent to the space of great circles. Using this fact as our starting point, an argument analogous to that of Section \ref{sec: round equiv smooth} yields the following.

\begin{thm}\label{thm: equiv great and smooth hl in s3}
The inclusion $\rlink(H,S^3)\hookrightarrow\link(H,S^3)$ is a homotopy equivalence. 
\end{thm}
\begin{proof}
    The proof that the relative homotopy groups $\pi_k(\link(H,S^3),\rlink(H, S^3))$ are trivial follows a very similar procedure to the proof of Theorem~\ref{abcthm: round space equivalence}, where we showed that the relative homotopy groups $\pi_k(\link(H),\rlink(H))$ are trivial. We highlight in the bullet points below the key inputs we use in the $S^3$ case, but the general set-up and the most technical part of the proof follows the proof of Theorem~\ref{abcthm: round space equivalence} without change.
    \begin{itemize}
        \item We first use (7) in \cite[Appendix]{Hatcher} to homotope the link component $L_1(t)$ to be a great circle for all $t\in D^k$.
        \item The round component $L_1(t)$ spans a family of flat discs in $S^3$. Make a choice of one of these discs to be $D_1(t)$ -- this can be done continuously over $D^k$. 
        \item To define an ellipsoid, note that there is a unique great circle perpendicular to the disc $D_1(t)$. We define the ellipsoid $E_h$ to have NP and SP on this great circle, at distance $h$ from $D_1(t)$, and choose the maximum height of the ellipsoid to be such that the upper and lower hemispheres of the ellipsoid are disjoint.
        \item The proof to homotope $D_2(t)$ continuously over $D^k$ such that it intersects $D_1(t)$ at its centrepoint is identical to that in the proof of Theorem~\ref{abcthm: round space equivalence}, except for the arc through the NP and SP of the ellipsoid we choose at one of the intermediate stages. In the $S^3$ setting, we homotope the arc in $E_i$ to the arc passing through the NP and SP which agrees with the segment of the great circle perpendicular to $D_1(t)$. This property will then pass to ellipsoids of smaller height, which we need to make the homotopy continuous over $D^k$.
        \item At this point instead of shrinking $L_1$ (which would mean it does not remain a great circle) we can consider the complement of the ball $B_i$ bounded by the ellipsoid $E_i$ in $S^3$, which is also a ball with boundary $E_i$. We can homotope the arc of $L_2(t)$ lying in this ball to the arc segment given by the the great circle perpendicular to $D_1(t)$. This can once again be done continuously over the $D^k$ since the great circle is the same for ellipsoids of all heights, and this completes the proof.\qedhere
    \end{itemize}
\end{proof}

We now show that $\rlink(H,S^3)$ has the homotopy type of a compact manifold, and thus determine the homotopy type of $\link(H,S^3)$ as well.

\begin{restate}{Theorem}{abcthm: S3 homotopy type}
    The embedding space $\link(H,S^3)$ is homotopy equivalent to $\R P^2\times \R P^2.$
\end{restate}

\begin{proof}
    By Theorem~\ref{thm: equiv great and smooth hl in s3} it is enough to show that the subspace $\rlink(H,S^3)$ is homeomorphic to $ \R P^2\times \R P^2$. Since each 2-plane in $\R^4$ has a unique orthogonal complementary 2-plane, each great circle $\gamma$ determines a \emph{unique} great  circle $\gamma^\perp$ linking $\gamma$ exactly once and lying in the orthogonal complement of the plane containing $\gamma$. Accordingly, the space of oriented, labeled, great Hopf links in $S^3$ with linking number $+1$ can be identified with the oriented Grassmannian \[\Gr_2^+(\R^4)=\SO(4)/(\SO(2)\times \SO(2)).\]
    Regarding $\SO(4)$ as orthonormal bases with determinant $+1$, we see that 
    \begin{enumerate}[(i)]
        \item $\SO(4)$ acts transitively on the space of oriented, labeled, great Hopf links.
        \item The identity coset corresponds to the basepoint $H$.
    \end{enumerate}
    
    As in Section \ref{sec: homotopy type}, to pass from oriented, labeled, great Hopf links to $\rlink(H,S^3)$, we must further quotient by the action of $\langle \alpha, s\rangle\cong \Z/2\times \Z/2$ where $\alpha$ acts by flipping over both link components, and $s$ acts by exchanging the link components. Accordingly, we see that taking the quotient by $\alpha$ yields $\Gr_2(\R^4)$, the space of unoriented 2-planes in $\R^4$. Quotienting further by $s$ then identifies each 2-plane $V\subset \R^4$ with its orthogonal complement $V^\perp$. Thus $\rlink(H,S^3)$ is the space of unordered pairs $\{V,V^\perp\}$ where $V\subset \R^4$ is 2-dimensional, or equivalently the space of all orthogonal splittings $\R^4=V\oplus V^\perp$.

To see that $\rlink(H,S^3)$ is $\R P^2\times \R P^2$, we first argue that $\Gr_2^+(\R^4)$ is homeomorphic to $\Sa^2\times \Sa^2$, closely following the exposition in \cite{Brasil-Ferreira-Vandembroucq}. Regarding $\R^4\cong \Hy$ as the quaternions, the standard inner product is given by $\langle x, y\rangle=\frac{1}{2}(x\bar{y}+y\bar{x}).$ The function $\mu(x,y)=\frac{1}{2}(x\bar{y}-y\bar{x})$ defines a map from $\Hy\times \Hy$ to the subset of pure imaginary quaternions, since  \[\overline{\mu(x,y)}=\frac{1}{2}\overline{(x\bar{y}-y\bar{x})}=\frac{1}{2}(y\bar{x}-x\bar{y})=-\mu(x,y).\]
Thus, when $x\perp y$ it follows from the inner product formula that $\mu(x,y)=y\bar{x}=-x\bar{y}.$
If, moreover, $x,y\in \Sa^3$ are orthogonal unit quaternions, then $\mu(x,y)\in \Sa^2$ is a purely imaginary unit quaternion. Similarly, the function $\nu(x,y)=\frac{1}{2}(\bar{x}y-\bar{y}x)$ defines a map from $\Hy\times\Hy$ to the subset of pure imaginary quaternions, and sends pairs of orthonormal quaternions to $\Sa^2$. 

Let $V_2(\R^4)=\{(x,y)\in \Sa^3\times \Sa^3\mid x\perp y\}$ be the Stiefel manifold of 2-frames in $\R^4$, and consider the map $\xi\colon V_2(\R^4)\to \Sa^2\times \Sa^2$ given by $\xi(x,y)=(\mu(x,y),\nu(x,y)).$ Since $\mu,\nu$ are clearly equivariant with respect to the $\SO(2)$ action given by right multiplication by $e^{i\theta}$, we obtain a quotient map  $\bar{\xi}\colon V_2(\R^4)/\SO(2)\cong \Gr^+_2(\R^4)\rightarrow \Sa^2\times \Sa^2.$ It is not difficult to see that $\bar{\xi}$ is a homeomorphism; see \cite{Brasil-Ferreira-Vandembroucq} for details.

Therefore, to compute the quotient $\rlink(H,S^3)$, we need only determine how $\alpha$ and $s$ act on $\Sa^2\times \Sa^2$. We may regard $\alpha,s$ as acting on $\SO(4)$ by right multiplication by the two matrices,  \[\alpha = \left(\begin{array}{cc|cc}1&0&0&0\\0&-1&0&0 \\\hline0&0&-1&0\\0&0&0&1\end{array}\right),
 ~s = \left(\begin{array}{cc|cc}0&0&1&0\\0&0&0&1 \\\hline1&0&0&0\\0&1&0&0\end{array}\right),\] hence the induced action on  $\Sa^2\times \Sa^2$ is by isometries with respect to the standard product metric. If $(x,y)\in V_2(\R^4)$ span a 2-plane, then $\alpha(x,y)=(x,-y)$. Thus $\mu(x,-y)=-\mu(x,y)$ and $\nu(x,-y)=-\nu(x,y)$, so $\xi(\alpha(x,y))=-\xi(x,y)$. It follows that $\alpha$ acts on $\Sa^2\times \Sa^2$ by the antipodal map $-I$ in both coordinates. Thus we conclude $\Gr_2(\R^4)=\Gr_2^+(\R^4)/\langle \alpha\rangle\cong\Sa^2\times \Sa^2/\langle (-I,-I)\rangle$. According to \cite[Ch.~12.2]{Hillman02}, the only group isomorphic to $\Z/2\times \Z/2$ that acts freely by isometries on $\Sa^2\times \Sa^2$ and contains $(-I,-I)$ is \[\langle(I,I),(-I,I),(I,-I),(-I,-I)\rangle.\]
Since $\langle \alpha,s\rangle$ acts freely, we conclude that $\rlink(H,S^3)$ is $\R P^2\times \R P^2$, as desired.
\end{proof}

\begin{rem}
    Theorem \ref{abcthm: S3 homotopy type} can also be derived from results of Havens--Koytcheff \cite[Corollary 4.4 (e)]{HavensKoytcheff2021}, who show that the parametrised embedding space of a Hopf link in $S^3$ is homotopy equivalent to $\SO(4).$
\end{rem}

We finish our discussion of Hopf links in $S^3$ by introducing an intermediate round embedding space. We show this space is homotopy equivalent to $\rlink(H,S^3)$ and hence $\link(H, S^3)$.

Let $\Sa^3$ denote $S^3$ equipped with its standard round metric. The set of conformal automorphisms of $\Sa^3$ is the group of \emph{M\"obius transformations} $\Mob(\Sa^3)$. While closed geodesics on $\Sa^3$ are exactly the great circles, one could also consider the larger collection of all circles which are the image of a great circle under a M\"obius transformation.

\begin{defn} By a \emph{round circle} in  $\Sa^3$ we will mean any circle which is in the orbit of the great circles under the action of $\Mob(\Sa^3)$. The space $\rlink^{\Mob}(H,S^3)$ of \emph{round Hopf links} is the subspace of $\link(H,S^3)$ where each component is a round circle. 
\end{defn}

The collection of round circles is exactly the set of circles which map to either a Euclidean circle or straight line under stereographic projection from some point of $\Sa^3$. This definition does not depend on which stereographic projection is used. In this sense $\rlink^\Mob(H,\Sa^3)$ is the natural analogue to the space $\rlink(H)$.

There is a chain of inclusions $\rlink(H,S^3)\hookrightarrow \rlink^\Mob(H,S^3)\hookrightarrow \link(H,S^3)$.

\begin{lem}\label{lem: greatcircles he round hl in s3}
    $\rlink(H,S^3)\hookrightarrow \rlink^{\mob}(H,S^3)$ is a homotopy equivalence. 
\end{lem}
\begin{proof}
    Let $\Hy^4$ denote real hyperbolic $4$-space in the disc model. By Poincar\'e extension, we can identify $\Mob(\Sa^3)$ with $ \Isom(\Hy^4)$. Let $\Mob^+(\Sa^3)\leq \Mob(\Sa^3)$ denote the index 2 subgroup of orientation-preserving M\"obius transformations, and hence orientation-preserving isometries of $\Hy^4$.  Round circles in $\Sa^3$ bound uniquely determined hyperbolic planes $\Hy^2\subset \Hy^4$, and two round circles link once exactly when the corresponding hyperbolic 2-planes meet transversely at a single point.

    Let $\prlink^\Mob(H,S^3)^+$ be the subspace of oriented, labeled, round Hopf links. We first define a deformation retraction of $\prlink^\Mob(H,S^3)^+$ onto a subspace $Y$ where the two components are orthogonal. As in the proof Lemma \ref{lem:PRHL Homotopy type}, this can be achieved by changing the angle to be $\pi/2$ in $\R^3$ under any stereographic projection, since stereographic projection is conformal. Since $\Mob^+(\Sa^3)$ acts transitively on the set of pairs of orthogonal, oriented hyperbolic planes $\Hy^2\subset \Hy^4$ with stabilisers isomorphic to $\Isom^+(\Hy^2)\times\Isom^+(\Hy^2)\cong \Mob^+(\Sa^1)\times\Mob^+(\Sa^1),$ we get an identification  $Y=\Mob^+(\Sa^3)/(\Mob^+(\Sa^1)\times\Mob^+(\Sa^1)).$ On the other hand, since $\SO(n)\hookrightarrow \Mob(\Sa^{n-1})$ is a homotopy equivalence and we have a commutative diagram of fiber sequences
    \[\xymatrix@R=.5cm{
        \Mob^+(\Sa^1)\times\Mob^+(\Sa^1)\ar[r]& \Mob^+(\Sa^3)\ar[r]& Y\\
        \SO(2)\times \SO(2)\ar[r]\ar[u]&\SO(4)\ar[r]\ar[u]&\Gr_2^+(\R^4)\ar[u]
    }\]
    we see that $\Gr_2^+(\R^4)\xrightarrow{\sim} Y$. Lastly, the deformation retraction onto $Y$ and the inclusion $\SO(4)\hookrightarrow \Mob(\Sa^3)$ are equivariant with respect to the  action of $\langle \alpha,s\rangle$, hence \[\rlink^\Mob(H,S^3)\simeq Y/\langle\alpha,s\rangle\cong\Gr_2^+(\R^4)/\langle \alpha,s\rangle=\R P^2\times \R P^2.\qedhere \] 
\end{proof}
 \bibliography{mybib}{}
 \bibliographystyle{alpha}
\end{document}